\documentclass{amsart}

\newcommand{\cO}{\mathcal{O}}

\newcommand{\Q}{\mathbb{Q}}
\newcommand{\Z}{\mathbb{Z}}

\newcommand{\m}{\to}

\DeclareMathOperator{\Id}{Id}
\DeclareMathOperator{\sgn}{sgn}
\DeclareMathOperator{\Vor}{Vor}
\DeclareMathOperator{\cH}{\mathcal{H}}
\providecommand{\mat}[1]{\begin{bmatrix} #1 \end{bmatrix}}

\newcommand{\cS}{\mathcal{S}}
\newcommand{\HH}{\mathbb{H}}
\newcommand{\OO}{\mathcal{O}}
\newcommand{\QQ}{\mathbb{Q}}
\newcommand{\ZZ}{\mathbb{Z}}
\newcommand{\CC}{\mathbb{C}}

\usepackage{booktabs}
\usepackage{microtype}
\usepackage{mathtools}

\usepackage[dvipsnames,svgnames,x11names]{xcolor}
\usepackage{url,graphicx,verbatim,amssymb,enumerate,stmaryrd, tikz, enumitem, multicol}
\usepackage[pagebackref,colorlinks,citecolor=blue,linkcolor=blue,urlcolor=blue,filecolor=blue]{hyperref}
\usetikzlibrary{matrix,arrows,decorations.pathmorphing}
\usetikzlibrary{arrows,decorations.pathmorphing,backgrounds,fit,positioning,shapes.symbols,chains,calc}
\usepackage{microtype}
\usepackage[all]{xy}
\usepackage[hmargin=3cm,vmargin=3cm]{geometry}
\usepackage{pdflscape}
\usepackage{aliascnt}
\usepackage{adjustbox}

\numberwithin{thmcounter}{section}
\newaliascnt{thmauto}{thmcounter}

\newaliascnt{Defauto}{thmcounter}

\newaliascnt{exauto}{thmcounter}

\newaliascnt{exsauto}{thmcounter}

\newaliascnt{lemauto}{thmcounter}

\newaliascnt{propauto}{thmcounter}

\newaliascnt{corauto}{thmcounter}

\newaliascnt{conjauto}{thmcounter}

\newaliascnt{queauto}{thmcounter}

\newaliascnt{remauto}{thmcounter}

\theoremstyle{plain}
\newtheorem{theorem}[thmauto]{Theorem}

\newtheorem{lemma}[lemauto]{Lemma}
\newtheorem{proposition}[propauto]{Proposition}
\newtheorem{corollary}[corauto]{Corollary}

\newtheorem*{rem*}{Remark}
\newtheorem*{thm*}{Theorem}
\newtheorem*{exs*}{Examples}

\theoremstyle{definition}
\newtheorem{definition}[Defauto]{Definition}

\newtheorem{remark}[remauto]{Remark}

\newcommand{\GL}{\mathrm{GL}}
\newcommand{\SL}{\mathrm{SL}}
\newcommand{\PSL}{\mathrm{PSL}}
\newcommand{\St}{\mathsf{St}}

\newcommand{\B}{\mathbf{B}}
\newcommand{\PB}{\mathbf{PB}}
\newcommand{\T}{\mathbf{T}}

\newcommand{\cB}{\mathcal{B}}
\newcommand{\cPB}{\mathcal{PB}}
\newcommand{\cT}{\mathcal{T}}

\newcommand{\Ab}{\underline{\mathrm{Ab}}}
\newcommand{\hgt}{\mathrm{ht}}
\newcommand{\rank}{\mathrm{rank}}

\title{Non-integrality of some Steinberg modules}

\author{Jeremy Miller}
\thanks{Jeremy Miller was supported in part by NSF grant DMS-1709726.}
\address{Department of Mathematics, Purdue University, USA}
\email{\href{mailto:jeremykmiller@purdue.edu}{jeremykmiller@purdue.edu}}
\author{Peter Patzt}
\address{Department of Mathematics, Purdue University, USA}
\email{\href{mailto:ppatzt@purdue.edu}{ppatzt@purdue.edu}}
\author{Jennifer C. H. Wilson}
\address{Department of Mathematics, University of Michigan, USA}
\email{\href{mailto:jchw@umich.edu}{jchw@umich.edu}}

\author{Dan Yasaki}
\address{Department of Mathematics and Statistics, The University of North Carolina Greensboro, USA}
\email{\href{mailto:d_yasaki@uncg.edu}{d\_yasaki@uncg.edu}}

\date{\today}

\begin{document}

\begin{abstract}	
We prove that the Steinberg module of the special linear group of a quadratic imaginary number ring which is not Euclidean is not generated by integral apartment classes. Assuming the generalized Riemann hypothesis, this shows that the Steinberg module of a number ring is generated by integral apartment classes if and only if the ring is Euclidean. We also construct new cohomology classes in the top dimensional cohomology group of the special linear group of some quadratic imaginary number rings.
\end{abstract}
	
\maketitle

\tableofcontents

\section{Introduction}

 The cohomology of arithmetic groups has many applications to number theory and
 algebraic $K$-theory. Let $\cO_K$ be the ring of integers in a
 number field $K$. One of the most useful tools for studying the high
 dimensional cohomology of $\SL_n(\cO_K)$  is the Steinberg module $\St_n(K)$, a representation
 of $\GL_n(K)$. Let $r_1$ denote the number of real embeddings of $K$, and let $r_2$ denote the number of pairs of complex embeddings.  Borel and Serre \cite{BoSe} proved that \[H^{\nu_n-i}(\SL_n(\cO_K);\Q) \cong H_i(\SL_n(\cO_K) ; \Q \otimes \St_n(K) )  \] with \[  \nu_n=r_1 \left( \frac{(n+1)n-2}{2}  \right)+ r_2 (n^2-1) - n+1.\] The number $\nu_n$ is the virtual cohomological dimension of $\SL_n(\cO_K)$, and all rational cohomology groups (even with twisted coefficients) vanish above this degree. 
 
 To understand the cohomology in degree $\nu_n$, it is important to understand generators for $\St_n(K)$ as an $\SL_n(\cO_K)$-module. There is a natural subset of $\St_n(K)$ consisting of so-called \emph{integral apartment classes} and it is useful to know whether these classes generate the entire Steinberg module. Let $\cT_n(K)$ denote the Tits building of $K$, that is, the geometric realization of the poset of proper non-empty subspaces of $K^n$ ordered by inclusion. The Steinberg module is defined by the formula \[ \St_n(K):=\widetilde H_{n-2}(\cT_n(K)). \] Let $\vec L$ denote a decomposition $K^n = L_1 \oplus \cdots \oplus L_n$ into lines. Let $A_{\vec L}$ be the full subcomplex of $\cT_n(K)$ with vertices direct sums of proper non-empty subsets of $\{L_1,\ldots, L_n \}$. Each subcomplex $A_{\vec L}$ is homeomorphic to an $(n-2)$-sphere called an \emph{apartment} and the images of the two choices of fundamental classes of $A_{\vec L}$ in $\St_n(K)=\widetilde H_{n-2}(\cT_n(K))$ are called \emph{apartment classes}. We say that an apartment or apartment class is \emph{integral} if \[ \cO_K^n = (\cO_K^n \cap L_1) \oplus \ldots \oplus (\cO_K^n \cap L_n). \]  In other words, $\vec L$ is integral if it comes from a direct sum decomposition of $\cO_K^n$ into rank-one projective submodules. 
 
One of the main topics of this paper is the question:
 \vspace{.1in}

\noindent \textbf{For what number fields $K$ is }$\St_n(K)$ \textbf{ generated by integral apartment classes?}  
 
 \vspace{.1in}
 
 \noindent Integral apartment classes vanish after taking coinvariants by $\SL_n(\cO_K)$ for $n$ larger than the class number of $\cO_K$, so if $\St_n(K)$ is generated by integral apartment classes, then $$H^{\nu_n}(\SL_n(\cO_K);\Q)  \cong H_0(\SL_n(\cO_K);\Q \otimes \St_n(K)   )$$ vanishes for $n$ sufficiently large. In \cite{AR}, Ash--Rudolph proved that $\St_n(K)$ is generated by integral apartment classes if $\cO_K$ is Euclidean.  For $n \geq 2$, Church--Farb--Putman proved that $H^{\nu_n}(\SL_n(\cO_K);\Q)$ does not vanish if the class number of $\cO_K$ is greater than $1$ \cite[Theorem D]{CFP-Integrality} and moreover that $\St_n(K)$ is not generated by integral apartment classes for $n \geq 2$ for such rings \cite[Theorem B]{CFP-Integrality}. We prove the following.

\begin{theorem} \label{main}
Let $\cO_K$ be a quadratic imaginary number ring that is a PID but is not Euclidean. Then $\St_n(K)$ is not generated by integral apartment classes.  
\end{theorem}

Let $\cO_d$ denote the ring of integers in $\Q(\sqrt d)$. The only examples of
rings satisfying the hypotheses of the above theorem are $\cO_d$ for $d=-19$,
$-43$, $-67$ and $-163$. However, assuming the generalized Riemann hypothesis, every number ring either has class number greater than $1$, is Euclidean, or is quadratic imaginary; see Weinberger \cite{Wein}.  Thus we have the following corollary.


\begin{corollary} \label{corollary}
Let $\cO_K$ be a ring of integers in a number field $K$ and consider $n \geq 2$. The generalized Riemann hypothesis implies that $\St_n(K)$ is generated by integral apartment classes if and only if $\cO_K$ is Euclidean.
\end{corollary}



For $K$ quadratic imaginary, $\nu_n=n^2-n$. Our proof of \autoref{main} also gives the following.
 
 \begin{theorem} \label{mainCohomology} For all $n$, we have \[\dim_{\Q} H^{\nu_{2n}}(\SL_{2n}(\cO_{d});\Q) \geq 
 \begin{cases}
1 & \text{ for $d=-43$}, \\
2^n & \text{ for $d=-67$}, \\
6^n & \text{ for $d=-163$.}
\end{cases} \]
 
  \end{theorem}
 
 This shows that the rational cohomological dimension and the virtual cohomological dimension agree for these groups. This is the first example of homology in the virtual cohomological dimension of $\SL_n(\cO_K)$ for large $n$ not coming from the class group. In particular, this gives the first example of the failure of \cite[Theorem D]{CFP-Integrality} to be sharp for large $n$. See \autoref{mapBounds} for a more conceptual description of these bounds.
 
 Our proof involves using cohomology classes in  $H^{\nu_{2}}(\SL_{2}(\cO_d);\Q)$ to construct classes in  $H^{\nu_{2n}}(\SL_{2n}(\cO_d);\Q)$. In particular, the inequalities in \autoref{mainCohomology} are actually equalities for $2n=2$ by the work of Rahm \cite[Proposition 1]{Rahm}.  This strategy does not apply to $ H^{\nu_{2n}}(\SL_{2n}(\cO_{-19});\Q)$ since $ H^{\nu_{2}}(\SL_{2}(\cO_{-19});\Q)$ vanishes. This also highlights the fact that our proof of non-integrality does not rely on homological non-vanishing. 
 
 \subsection*{Acknowledgments}
 We thank Alexander Kupers, Ben McReynolds, Rohit Nagpal,  Andrew Putman, and the anonymous referee for helpful comments. 

\section{Posets}

In this section, we begin by fixing notation for posets. Then we recall some connectivity results for complexes of unimodular vectors and their variants. 

\subsection{Notation} \hfill

\begin{definition} Given a simplicial complex $\mathcal{X}$, there is an associated poset $\mathbf{X}$ whose elements are the simplices of $\mathcal{X}$, ordered by inclusion. \end{definition} 
 
 In this paper we adopt the convention that we use calligraphic fonts for simplicial complexes and boldface fonts for their associated posets.

 \begin{definition}
 Given a poset $\mathbf{Y}$, let $\Delta(\mathbf Y)$ denote the simplicial complex of non-degenerate simplices of the nerve of the poset.  \end{definition} 
 
Concretely, the $p$--simplices of $\Delta(\mathbf Y)$ are given by ordered $(p+1)$--tuples \[ \{ y_0 < y_1 < \cdots <y_p \; | \; y_i \in \mathbf{Y} \}.\] 
Define the \emph{dimension} $\dim(\mathbf{Y})$ of $\mathbf{Y}$ to be the dimension of $\Delta(\mathbf{Y})$.  Let $|\mathbf{Y}|$ denote the geometric realization of $\Delta(\mathbf{Y})$. We refer to the \emph{connectivity} of a poset or simplicial complex to mean the connectivity of its geometric realization.

We remark that, given a simplicial complex  $\mathcal{X}$ with associated poset  $\mathbf{X}$, $\Delta(\mathbf{X})$ is the barycentric subdivision of $\mathcal{X}$. Thus they are not isomorphic in general but have homeomorphic geometric realizations.

 \begin{definition} Given a poset $\mathbf{Y}$ and $y \in \mathbf{Y}$, let \[ \mathbf{Y}_{\leq y} := \{ y' \in \mathbf{Y} \; | \; y' \leq y \}\] and \[ \mathbf{Y}_{> y} := \{ y' \in \mathbf{Y} \; | \; y' > y \}.\]  Define the \emph{height} of $y$ to be \[\hgt(y) := \dim(\mathbf{Y}_{\leq y}).\] This is one less than the length of a maximal length chain with supremum $y$.  \end{definition}

 \subsection{The complex of partial frames} \hfill
 
 In this subsection, we describe the complex of partial bases and a variant called the complex of partial frames. Here and in the rest of the paper, the symbol $R$ will denote a commutative ring. 

\begin{definition} For a finite-rank free $R$--module $V$, we associate a simplicial complex $\cPB(V)$ called the  \emph{complex of partial bases} of $V$. The vertices of $\cPB(V)$ are primitive vectors in $V$, and vertices $v_0, v_1, \ldots, v_p$ span a $p$--simplex if and only if the vectors  $v_0, v_1, \ldots, v_p$ are a \emph{partial basis} for $V$, that is, a subset of a basis. When $V=R^n$, we will abbreviate this by $\cPB_n$ or by $\cPB_n(R)$ when we want to emphasize the ring.  We write $\PB(V)$, $\PB_n$, or  $\PB_n(R)$ to denote the posets associated to these simplicial sets; these are the posets of partial bases under inclusion. 
\end{definition}

Note that the complex $\cPB(V)$  has dimension $(\rank(V)-1)$. 


\begin{definition} For $V$ a finite-rank free $R$-module, we write $\cB(V)$ (similarly $\cB_n$, $\cB_n(R)$)  for the simplicial complex defined as the quotient of $\cPB(V)$ by identifying vertices $v, u$ if the vectors $v,u \in V$ differ by multiplication by a unit. A $p$--simplex of  $\cB(V)$ encodes a decomposition of a direct summand of $V$ into a direct sum of $(p+1)$ rank-one free submodules of $V$. Following Church--Putman \cite{CP}, we call such a simplex a \emph{partial frame} and $\cB(V)$ the \emph{complex of partial frames} of $V$. We write  $\B(V)$, $\B_n$, or $\B_n(R)$, respectively, for the associated posets. 
\end{definition} 


\begin{definition} For $V$ a finite-rank free $R$-module, we write $\cB'(V)$ (similarly $\cB'_n$, $\cB'_n(R)$)  for the $(\rank(V)-2)$--skeleton of $\cB(V)$, and $\B'(V)$, $\B'_n$, or $\B'_n(R)$, respectively, for the associated posets. 
\end{definition} 

Simplices of $\B'(V)$ consist of partial frames whose sum is not all of $V$. 

\begin{definition} Let $R$ be an integral domain. For a finite-rank free $R$--module $V$, we write $\T(V)$ (similarly $\T_n$ or $\T_n(R)$) for  the poset of proper non-zero direct summands of $V$ ordered by inclusion. The associated simplicial complex is the \emph{Tits building} for $V$. We abbreviate $\Delta(\T(V))$, $\Delta(\T_n)$, and $\Delta(\T_n(R))$ by $\cT(V)$, $\cT_n$, and $\cT_n(R)$ respectively.

\end{definition}

\begin{remark} \label{RemarkFrac} Let $V$ be a finite-rank free $R$--module and $R$ a Dedekind domain. Since $R$ is an integral domain, it embeds in its field of fractions $\mathrm{Frac}(R)$. There is a natural bijection between the direct summands of $V$ and the subspaces of the $\mathrm{Frac}(R)$--vector space $\mathrm{Frac}(R) \otimes_R V$ given by sending a submodule of $R^n$ to its $\mathrm{Frac}(R)$-span in $\mathrm{Frac}(R) \otimes_R V$. This map has an inverse given by intersection with $R^n$. It is an elementary exercise
to check that these maps are inverses using the following property of finitely-generated modules over Dedekind domains: a submodule $M$ of a finitely-generated module $V$ is a summand if and only if $V/M$ is torsion-free. This bijection induces a natural isomorphism $\T(V) \cong \T(\mathrm{Frac}(R) \otimes_R V)$, where $V$ is viewed as an $R$-module and $\mathrm{Frac}(R) \otimes_R V$ is viewed as a $\mathrm{Frac}(R)$-module. 
\end{remark}

 \subsection{Connectivity results} \hfill

The following is known as the Solomon--Tits Theorem. It seems to have first appeared in print in Solomon \cite{Solomon} in the case of finite fields. The general case appears in Garland \cite[Theorem 2.2]{Garland} and Quillen \cite[Theorem 2]{Quillen-Ki}. 


\begin{proposition}[{\bf Solomon--Tits Theorem}] \label{SolomonTits} Let $K$ be a field.  For $n \geq 2$, $|\cT_n(K)|$ has the homotopy type of a wedge of $(n-2)$--spheres.
\end{proposition} 

By \autoref{RemarkFrac},  $\cT_n(R)$ is isomorphic to $\cT_n(\mathrm{Frac}(R))$ when $R$ is a Dedekind domain and hence it is also $(n-3)$--connected.

The following is straightforward and is the reason we primarily use $\cB_n$ instead of $\cPB_n$. 

\begin{proposition} \label{B1contractible}
Let $R$ be a PID. Then $\cB_1(R)$ is contractible. 
\end{proposition}

\begin{proposition} \label{B2ConnectedElementary}  If $\GL_2(R)$ is not generated by matrices of the
  form \[\left\{ \begin{bmatrix} u & * \\ 0 & v \end{bmatrix},  \begin{bmatrix}
    u & 0  \\ * & v \end{bmatrix},  \begin{bmatrix} 0 & u  \\ v &
    0 \end{bmatrix} \;  \middle|  \; \text{ $u,v \in R^{\times}$, $* \in R$}
  \right\},\] 
 then the graph $\cB_2(R)$ is
  not connected. 
\end{proposition}


\begin{proof}

Recall that a simplicial complex of dimension $d$ is called Cohen--Macaulay if it is $(d-1)$-connected and if the links of $p$-simplicies are $(d-2-p)$-connected. In particular, a graph of dimension $1$, that is a graph with at least one edge, is Cohen–Macaulay if and only it is connected.

Church--Farb--Putman \cite[proof of Theorem 2.1]{CFP-Integrality} proved that the graph $\cPB_2(R)$ is not Cohen--Macaulay  if $\GL_2(R)$ is not generated by elementary matrices together with diagonal matrices. Let $S_1$ be the set of elementary matrices together with diagonal matrices and let $S_2$ be the set of matrices of the form \[\left\{ \begin{bmatrix} u & * \\ 0 & v \end{bmatrix},  \begin{bmatrix}
    u & 0  \\ * & v \end{bmatrix},  \begin{bmatrix} 0 & u  \\ v &
    0 \end{bmatrix} \;  \middle|  \; \text{ $u,v \in R^{\times}$, $* \in R$}
  \right\}.\] We will show that $S_1$ and $S_2$ generate the same subgroup of $\GL_2(R)$. Since $S_2$ contains $S_1$, it suffices to show that every matrix in $S_2$ can be written as a product of matrices in $S_1$. To show this, we perform the following elementary calculation: 
  \[ \begin{bmatrix} u & * \\ 0 & v \end{bmatrix}=  \begin{bmatrix}
    1 &  */v  \\0 & 1 \end{bmatrix} \begin{bmatrix}
    u& 0  \\ 0 & v \end{bmatrix}, \, \begin{bmatrix} u & 0 \\ * & v \end{bmatrix}=  \begin{bmatrix}
    1 &  0 \\ */u  & 1 \end{bmatrix} \begin{bmatrix}
    u& 0  \\ 0 & v \end{bmatrix}, \] \[\text{and } \begin{bmatrix} 0 & u \\ v & 0 \end{bmatrix}= \begin{bmatrix} 1 & -1 \\ 0 & 1 \end{bmatrix} \begin{bmatrix} 1 & 0 \\ 1 & 1 \end{bmatrix} \begin{bmatrix} 1 & -1 \\ 0 & 1 \end{bmatrix} \begin{bmatrix} 1 &0 \\ 0 & -1 \end{bmatrix}    \begin{bmatrix} v & 0  \\0 & u \end{bmatrix}.\]

   From now on, assume that $\GL_2(R)$ is not generated by these matrices and so $\cPB_2(R)$ is not Cohen--Macaulay. Consider the natural map $\cPB_2(R) \m \cB_2(R)$. This map is a complete join complex in the sense of Hatcher--Wahl \cite[Definition 3.2]{hatcherwahl}. Thus, by \cite[Proposition 3.5]{hatcherwahl},  $\cB_2(R)$ is not Cohen--Macaulay. Because $\cB_2(R)$ has at least one edge, it cannot be connected. \end{proof}

Combining  \autoref{B2ConnectedElementary} with Cohn \cite[Theorem 6.1 and 6.2]{Cohn} implies the following result.

\begin{proposition} \label{B2notCon}
Let $\cO_K$ be a quadratic imaginary number ring that is a PID but is not Euclidean. Then $\cB_2(\cO_K)$ is not connected. 
\end{proposition}

\begin{proposition} \label{split} If $\cPB_n(R)$ is $d$--connected for some ring
  $R$, then so is $\cB_n(R)$.
\end{proposition}
\begin{proof} Consider the natural projection 
$\cPB_n(R) \m \cB_n(R)$. We can construct a splitting $\cB_n(R) \m \cPB_n(R)$ by choosing a primitive vector $v$ for each line in $R^n$. Hence $\pi_i(|\cPB_n(R)|) \m \pi_i(|\cB_n(R)|)$ is surjective for all $i$.
\end{proof}

Combining this with a result of van der Kallen \cite[Theorem 2.6 (i)]{vdK} gives the following corollary. 

\begin{corollary}\label{vdKB}
Let $R$ be a PID. Then $\cB_n(R)$ is $(n-3)$--connected. 
\end{corollary}

\begin{proof} 
Since PIDs satisfy the Bass Stable Range condition SR$_3$, the work of van der Kallen \cite[Theorem 2.6 (i)]{vdK} implies $\cPB_n(R)$ is $(n-3)$--connected. The claim now follows from \autoref{split}.
\end{proof}

\begin{corollary} \label{primecon}
Let $R$ be a PID. Then $\cB_n'(R)$ is $(n-3)$--connected. 
\end{corollary}

\begin{proof}
The complex $\cB_n'(R)$ is the $(n-2)$--skeleton of $\cB_n(R)$, which is $(n-3)$--connected. The claim follows from simplicial approximation. 
\end{proof}

\section{The map of posets spectral sequence}

In this section we recall a useful spectral sequence arising from maps of posets. 

\subsection{Homology of posets} \hfill

 
We begin by defining the homology of a poset with coefficients in a functor $F$. 
 
  \begin{definition}  Given a poset $\mathbf{Y}$, and a functor $F$ from the poset $\mathbf{Y}$ (viewed as a category) to the category $\Ab$  of abelian groups, define the chain groups
  \[C_p(\mathbf{Y}; F) := \bigoplus_{y_0<\cdots<y_p \in \mathbf{Y}} F(y_0)\]
  and differential given by the alternating sum of the face maps
  \begin{align*}  & d_i\colon \bigoplus_{y_0<\cdots<y_p } F(y_0) \longrightarrow \bigoplus_{y_0<\cdots < \hat{y_i}< \cdots <y_p } F(y_0)  \qquad \qquad (0<i \leq p) \\
  & d_0\colon \bigoplus_{y_0<\cdots<y_p } F(y_0) \longrightarrow \bigoplus_{y_1< \cdots <y_p } F(y_1).  
  \end{align*}
  Here, the map $d_i$ with $(i \neq 0)$ maps the summand indexed by $(y_0<\cdots<y_p)$ to the summand indexed by $(y_0<\cdots < \hat{y_i}< \cdots <y_p)$, and acts by the identity on the group $F(y_0)$. The map $d_0$ maps the summand indexed by $(y_0<\cdots<y_p)$ to the summand indexed by $(y_1< \cdots <y_p)$, and the map $F(y_0) \to F(y_1)$ is the image of the morphism $y_0<y_1 \in \mathbf{Y}$ under the functor $F$.    \end{definition}
   
If  $F=\Z$ is the constant functor with identity maps, then $H_*(\mathbf{Y}; \Z)$ coincides with the homology groups $H_*(|\mathbf{Y}|)$. 

The following lemma is adapted from Charney \cite[Lemma 1.3]{Charney-Generalization}. 
   
   \begin{lemma} \label{ChainsSingleSupport} Suppose that $F\colon \mathbf{Y} \to \Ab$ is a functor supported on elements of height $m$. Then 
   \[H_p(\mathbf{Y};F) = \bigoplus_{\hgt(y_0)=m} \widetilde{H}_{p-1}(\mathbf{Y}_{>y_0}; F(y_0)).\]
   \end{lemma}
   
   \begin{proof} Suppose that $F\colon \mathbf{Y} \to \Ab$ is supported on elements of height $m$. \\[-1.5em]
   \begin{align*}
  C_p(\mathbf{Y}; F) = \bigoplus_{y_0<\cdots<y_p \in \mathbf{Y}} F(y_0) &\cong \bigoplus_{\hgt(y_0)=m} \left( F(y_0) \otimes_{\Z} \bigoplus_{y_0<\cdots<y_p} \Z \right) \\
  = \bigoplus_{\substack{y_0<\cdots<y_p \in \mathbf{Y} \\ \hgt(y_0)=m}} F(y_0) 
  &  \cong \bigoplus_{\hgt(y_0)=m} \Big( F(y_0) \otimes_{\Z} \widetilde{C}_{p-1}(\mathbf{Y}_{>y_0}; \Z) \Big).
   \end{align*}
The composition of these isomorphisms is  compatible with the
   differentials and hence gives an isomorphism of chain
   complexes. Thus \[H_p(\mathbf{Y}; F) = \bigoplus_{\hgt(y_0)=m}
   \widetilde{H}_{p-1}(\mathbf{Y}_{>y_0};  F(y_0)). \qedhere\]
   \end{proof}
   

\subsection{The spectral sequence for a map of posets}

\hfill 

Given a map of posets $f\colon \mathbf{X} \to \mathbf{Y}$, there is an associated spectral sequence introduced by Quillen \cite[Section 7]{Quillen-Poset}; see also Charney \cite[Section 1]{Charney-Generalization}.

\begin{definition}  Let $f\colon \mathbf{X} \to \mathbf{Y}$ be a map of posets.  For $y \in \mathbf{Y}$, define $f\backslash y \subseteq \mathbf{X}$ to be the subposet of elements whose images in $\mathbf{Y}$ are less than or equal to $y$: \[f\backslash y := \{ x \in \mathbf{X} \; | \; f(x) \leq y \}.\]
\end{definition}

\begin{theorem}Given a map of posets  $f\colon \mathbf{X} \to \mathbf{Y}$, there is a spectral sequence 
\[ E^2_{p,q} = H_p\Big(\mathbf{Y};  [y \mapsto H_q(f \backslash y)]\Big) \qquad \implies \qquad H_{p+q}(\mathbf{X}).  \]
\end{theorem}

This spectral sequence is an instance of the Grothendieck spectral sequence for the composition of functors:
    \begin{align*}
 \mathrm{Fun}( \mathbf{X} , \Ab) &\longrightarrow \mathrm{Fun}( \mathbf{Y}, \Ab) \\ 
F & \longmapsto [y \mapsto H_0(f \backslash y; F)] \\[.5em]
\mathrm{Fun}( \mathbf{Y}, \Ab) &\longrightarrow \Ab \\ 
F' & \longmapsto H_0( \mathbf{Y}, F'). 
\end{align*}

    \section{Non-integrality} 
\label{sec4}
In this section, we use the map-of-posets spectral sequence and connectivity results to prove that the Steinberg modules of quadratic imaginary PIDs which are not Euclidean are not generated by integral apartment classes. We begin by relating the complex of partial frames to integral apartment classes. 

Note that $H_{n-1}(\cB_n,\cB_n')$ is the free abelian group on the set of $(n-1)$-simplices of $\cB_n$. In other words, $H_{n-1}(\cB_n,\cB_n')$ is isomorphic to the quotient of the free abelian group on symbols $(F_1,\ldots,F_n)$ with $F_i$ rank-one free submodules of $R^n$, with $R^n = F_1 \oplus \cdots \oplus F_n$, modulo the relation \[(F_1,\ldots,F_n)=\sgn(\sigma)(F_{\sigma(1)},\ldots, F_{\sigma(n)}), \qquad \qquad \sigma \text{ a permutation of $n$, $\sgn(\sigma)$ the sign of $\sigma$}.\] There is a map $\alpha\colon H_{n-1}(\cB_n,\cB_n') \m H_{n-2}(\T_n)=\St_n(R)$ sending $(F_1,\ldots,F_n)$ to \[\sum_{\sigma \in S_n} \sgn(\sigma) \Big(F_{\sigma(1)} \subset F_{\sigma(1)} \oplus F_{\sigma(2)} \subset \cdots \subset F_{\sigma(1)} \oplus \cdots \oplus F_{\sigma(n-1)} \Big).\] The image of $\alpha$ is the submodule of $\St_n(R)$ generated by integral apartment classes. In particular, $\alpha$ is surjective if and only if the Steinberg module is generated by integral apartment classes.

For the rest of this section, we will  study the spectral sequence associated to the map of posets
\begin{align*}
f\colon \B'_n & \longrightarrow \T_n \\ 
\{ v_0,  \ldots,  v_p\} & \longmapsto  \mathrm{span}_R(v_0, \ldots, v_p).
\end{align*}    
Throughout the section we let $E^r_{p,q}$ denote this  spectral sequence (with implicit dependence on a fixed choice of $n$).  

Observe that, for $V \in \T_n$, the subposet $f\backslash V$ is precisely $\B(V)$. Thus the spectral sequence associated to $f$ satisfies
\[ E^2_{p,q} = H_p\Big(\T_n;  [V \mapsto H_q(\B(V))]\Big) \qquad \implies \qquad  H_{p+q}(\B'_n).  \]

We will use the following lemma to further describe the $E^2$ page of the spectral  sequence in \autoref{E2/Einfty}. 

   \begin{lemma}\label{H(T_n,tilde H_0)}  Let $R$ be a PID. Then $H_{p}(\T_n; \widetilde H_0(\B(-))) \cong 0$ unless $p = n-3$, when
\[ H_{n-3}(\T_n; \widetilde H_0(\B(-))) \cong  \bigoplus_{\substack{V \subseteq R^n \\ \rank(V)=2}} \St_{n-2} \otimes \widetilde H_0(\B_2).\]
In particular, if $\B_2$ is not connected and $n \geq 3$, the group $ H_{n-3}(\T_n; \widetilde H_0(\B(-))) $ is not zero.
\end{lemma}

\begin{proof}
The functor $\widetilde H_0(\B(-)) \colon \T_n \m \Ab$ is zero except possibly on submodules $V \subseteq R^n$ of rank $2$ by \autoref{B1contractible} and \autoref{vdKB}. Then by \autoref{ChainsSingleSupport}, we find
\[H_{p}(\T_n; \widetilde H_0(\B(-))) =  \bigoplus_{\substack{V \subseteq R^n \\ \rank(V)=2}} \widetilde{H}_{p-1}\left(\T(R^n/V);  \widetilde{H}_0(\B(V))  \right). \] 
Then $\T(R^n/V)$ is spherical of dimension $(n-2)-2$ by the Solomon--Tits theorem, and we find 
\[H_{p}(\T_n; \widetilde H_0(\B(-))) = 
\begin{cases}
  \bigoplus_{\substack{V \subseteq R^n \\ \rank(V)=2}}
  \St(R^n/V) \otimes  \widetilde H_0(\B(V)) & \text{if $p=n-3$,} \\ 
0 & \text{otherwise}.  \end{cases} \]
In order to see that $H_{n-3}(\T_n; \widetilde H_0(\B(-))) $ is not zero for $n\ge 3$, recall that $\St_n$ is non-zero for $n\ge 1$. In particular, $\T_1$ is empty so $\St_1 \cong \widetilde H_{-1}(\T_1) \cong \Z$.
\end{proof}

We now establish some features of the spectral sequence $E^r_{p,q}$. 

\begin{proposition} \label{E2/Einfty} Let $n \geq 3$, $R$ be a PID, and $E^r_{p,q}$ denote the spectral sequence associated to the map of posets $f\colon \B'_n \to \T_n$. Then 
\begin{enumerate}[label=(\roman*)]
\item \label{E-infty}  $ E^\infty_{p,q} \cong 0$ unless $p+q = n-2$ or $(p,q)=(0,0)$, 
\item \label{E2} $ E^2_{p,q} \cong 0$ unless $p+q = n-2$, $p+q = n-3$, or $(p,q)=(0,0)$,
\item \label{E2corner} $E^2_{n-3,0} \cong 0$ when $n>3$, and $E^2_{n-3,0} \cong \Z$ when $n=3$.
\end{enumerate}
\end{proposition}

The spectral sequence is illustrated in \autoref{FigureE2/Einfty}. 

 \begin{figure}[h!]    \hspace{-0cm} \hspace{-1.2cm} 
\begin{center}  \begin{tikzpicture} \scriptsize
  \matrix (m) [matrix of math nodes,
    nodes in empty cells,nodes={minimum width=3ex,
    minimum height=3ex,outer sep=2pt},
 column sep=3ex,row sep=3ex]{  
  6    & 0 &  &  & & &  &  \\   
  5    &\bigstar & 0 && &  &  &   \\ 
 4    & \bigstar & \bigstar & 0 &  &  &   &  \\   
 3    & 0& \bigstar & \bigstar & 0 & &   & \\  
 2   & 0 & 0 & \bigstar & \bigstar & 0 &  &  \\            
1     & 0 & 0 & 0 & \bigstar& \bigstar & 0 &  \\        
 0     & \bigstar & 0 & 0 & 0&|[draw=red, circle]|0& \bigstar& 0  \\       
 \quad\strut &     0  &  1  & 2  & 3 & 4 &5 &6  \\}; 
 \draw[thick] (m-1-1.east) -- (m-8-1.east) ;
 \draw[thick] (m-8-1.north) -- (m-8-8.north east) ;

    \begin{pgfonlayer}{background}
\draw[rounded corners, draw=none, fill=black!60!white, inner sep=3pt,fill opacity=0.25]
   (m-3-2.north)  -| (m-3-2.south east) 
    -| (m-4-3.north east) -| (m-4-3.south east)
    -| (m-5-4.north east) -| (m-5-4.south east)
     -| (m-6-5.north east) -| (m-6-5.south east)
     -| (m-7-6.north east) -| (m-7-6.south east)
      -| (m-7-3.north west)
      -| (m-3-2.north west) -| (m-3-2.north);
    \end{pgfonlayer}

        \begin{pgfonlayer}{background}
\draw[rounded corners, draw=none, fill=black!60!white, inner sep=3pt,fill opacity=0.25]
   (m-1-2.north) -| (m-1-8.north)
      -| (m-7-8.north east) -| (m-7-8.south east)
       -| (m-7-8.south west) -| (m-7-8.north west)
       -| (m-6-7.south west) -| (m-6-7.north west)
     -| (m-5-6.south west) -| (m-5-6.north west)
      -| (m-4-5.south west) -| (m-4-5.north west)
        -| (m-3-4.south west) -| (m-3-4.north west)
               -| (m-2-3.south west) -| (m-2-3.north west)
                      -| (m-1-2.south west) -| (m-1-2.north west)
       ;
    \end{pgfonlayer}

\end{tikzpicture}
\end{center}
\caption{The page $E^2_{p,q}$ when $n=7$. The term $E^2_{n-3,0}$ is circled in red. Terms that are shaded grey converge to zero.}
\label{FigureE2/Einfty}
\end{figure}
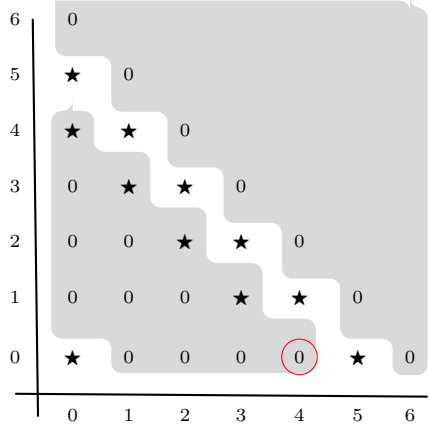

\begin{proof} Since the spectral sequence converges to $H_{p+q}(\B'_n)$ and $\dim(\B'_n) = n-2$, part \ref{E-infty} follows from the fact that $\B_n'(R)$ is $(n-3)$--connected by \autoref{primecon}. Part \ref{E2corner} follows from parts \ref{E-infty}  and \ref{E2}:  part \ref{E2} implies that  for $r\geq 2$ there are no non-trivial differentials to or from the group $E^r_{n-3, 0}$ (see \autoref{E^2_{n-3,0}}), so $E^2_{n-3, 0} = E^{\infty}_{n-3, 0}$.   
 \begin{figure}[h!]    \hspace{-0cm} \hspace{-1.2cm} 
\begin{tikzpicture} \scriptsize
  \matrix (m) [matrix of math nodes,
    nodes in empty cells,nodes={minimum width=3ex,
    minimum height=5ex,outer sep=2pt},
 column sep=3ex,row sep=3ex]{  
 3    & 0& \bigstar & \bigstar & 0 & &   && \\  
 2   & 0 & 0 & \bigstar & \bigstar & 0 &  &&  \\            
1     & 0 & 0 & 0 & \bigstar& \bigstar & 0 &&  \\        
 0     & \bigstar & 0 & 0 & 0& \bigstar & \bigstar& 0&  \\       
 \quad\strut &     0  &  1  & 2  & 3 & 4 &5 &6 &  \\ &&&&&&&& \\}; 
 \draw[thick] (m-1-1.east) -- (m-5-1.east) ;
 \draw[thick] (m-5-1.north) -- (m-5-8.north east) ;
 
 \draw[-stealth, blue]  (m-4-6) -- (m-3-4);
  \draw[-stealth, blue]  (m-4-6) -- (m-2-3);
    \draw[-stealth, blue]  (m-4-6) -- (m-1-2);
  \draw[-stealth, blue]  (m-5-8) -- (m-4-6);
    \draw[-stealth, blue]  (m-6-9) -- (m-4-6);

\end{tikzpicture}
\vspace{-2em}
\caption{$E^r_{n-3,0}$ admits no non-trivial differentials for page $r \geq 2$, illustrated when $n=7$.}
\label{E^2_{n-3,0}}
\end{figure}
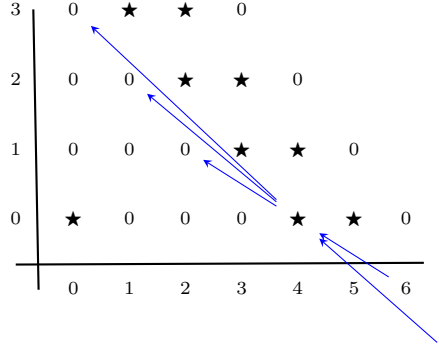

By part \ref{E-infty}  we conclude that $E^2_{n-3,0} \cong 0$  for all $n> 3$, and when $n=3$ we see $$E^2_{n-3, 0} = E^2_{0, 0} = E^{\infty}_{0, 0} = H_{0}(\B'_n) \cong \Z,$$ where the final isomorphism follows from \autoref{primecon}.

It remains to show part \ref{E2}, which we do in two parts: we first treat the case $q>0$, and then the case $q=0$. For $q>0$ the groups $H_q(\B(V))$ are non-zero only when $\rank(V)$ is $(q+1)$ or $(q+2)$ by \autoref{vdKB}. We can therefore realize the functor $H_q(\B(-))$ as an extension of functors $F''$ by $F'$ each supported on elements $V$ of a single height, as follows: 

\begin{center} \footnotesize
\begin{tikzpicture} 
\matrix(m)[matrix of math nodes,
row sep=.8em, column sep=2em, inner sep=1em,
text height=1.5ex, text depth=0.25ex]
{ &\rank(U)=q+1  & \rank(W)=q+2 &  & 0&   \\ 
&0 & H_{q} (\B(W)) &   & F'= \left[ V \longmapsto \left\{ \begin{array}{l}H_{q} (\B(V)),  \rank_R(V)=q+2 \\ 0,  \qquad \qquad\text{otherwise} \end{array}\right. \right]  & \\
 &H_{q} (\B(U))  & H_{q} (\B(W)) &  & H_q(\B(-))& \\
 &H_{q} (\B(U)) & 0 & & F''= \left[ V \longmapsto \left\{ \begin{array}{l}  H_{q} (\B(V)),  \rank_R(V)=q+1 \\ 0,  \qquad \qquad\text{otherwise} \end{array} \right. \right]   &   \\
& & &  & 0&   \\ 
};
\path[->,font=\scriptsize,>=angle 90]
(m-2-2) edge (m-2-3)
(m-3-2) edge node [above] {$(U \hookrightarrow W)_*$}  (m-3-3)  
(m-4-2) edge (m-4-3)
(m-2-2) edge (m-3-2)
(m-3-2) edge (m-4-2)
(m-2-3) edge (m-3-3)
(m-3-3) edge (m-4-3)
(m-1-5) edge (m-2-5)
(m-2-5) edge (m-3-5)
(m-3-5) edge (m-4-5)
(m-4-5) edge (m-5-5)
;
\end{tikzpicture} \end{center} We can then apply \autoref{ChainsSingleSupport} to the terms in the associated long exact sequence on homology:  \begin{center}
\begin{tikzpicture} \scriptsize
\matrix(m)[matrix of math nodes,
row sep=1em, column sep=1.6em,
text height=1.5ex, text depth=0.25ex]
{ \cdots &H_p(\T_n;  F') & H_p\Big(\T_n,  [V \mapsto H_q(\B(V))]\Big) & H_p(\T_n;  F'') & \cdots \\ 
 \; &\displaystyle  \bigoplus_{\substack{ W \subseteq R^n \\ \rank(W)=q+2}} \widetilde{H}_{p-1}(\T(R^n/W); H_q(\B(W))  & E^2_{p,q}  & \displaystyle   \bigoplus_{\substack{ U \subseteq R^n \\ \rank(U)=q+1}} \widetilde{H}_{p-1}(\T(R^n/U); H_q(\B(U))  & \\  &\\
};
\path[->,font=\scriptsize,>=angle 90]
(m-1-1) edge (m-1-2)
(m-1-2) edge (m-1-3)
(m-1-3) edge (m-1-4)
(m-1-4) edge (m-1-5);
\draw[double]
(m-1-2) -- (m-2-2)
(m-1-3) -- (m-2-3)
(m-1-4) -- (m-2-4);
\end{tikzpicture} \end{center} Since the reduced homology of $\T(V)$ is supported in degree $\left(\rank(V)-2\right)$ by \autoref{SolomonTits}, we conclude from this long exact sequence that for $q>0$ the homology groups $E^2_{p,q}$ can be non-zero only when $(p+q)$ is equal to $(n-3)$ or $(n-2)$. 

Now, consider the case when $q=0$. The homology group $H_0(\B(V))$ is $\Z$ for $\rank(V) \neq 2$ by \autoref{B1contractible} and \autoref{vdKB}. Thus we can express the functor $H_0(\B(-))$ as an extension of the constant functor $\Z$ by the functor $\widetilde H_0(\B(-))$ supported on submodules $V$ of rank $2$. 

\begin{center}
\begin{tikzpicture}  \footnotesize
\matrix(m)[matrix of math nodes, 
row sep=1.6em, column sep=1.6em, inner sep=1em,
text height=1.5ex, text depth=0.25ex]
{ & & &  & 0&   \\ 
 &   \left[ V \longmapsto  \widetilde{H}_0(\B(V)) = \left\{ \begin{array}{l} 0,  \qquad \quad \rank(V) \neq 2 \\  \widetilde{H}_0(\B_2),  \rank(V) =2 \end{array} \right. \right]  & &  & \widetilde{H}_0(\B(-)) & \\
 & \left[V \longmapsto H_0(\B(V)) \right] & &  & H_0(\B(-))& \\
 & \left[V \longmapsto \Z \right] & &  & \Z  &  \\
& & &  & 0&   \\ 
};
\path[->,font=\scriptsize,>=angle 90]
(m-2-2) edge (m-3-2)
(m-3-2) edge (m-4-2)
(m-1-5) edge (m-2-5)
(m-2-5) edge (m-3-5)
(m-3-5) edge (m-4-5)
(m-4-5) edge (m-5-5)
;
\end{tikzpicture} \end{center}

We apply  \autoref{ChainsSingleSupport} and \autoref{H(T_n,tilde H_0)} to the associated long exact sequence on homology groups:
\begin{center}
\begin{tikzpicture} \scriptsize
\matrix(m)[matrix of math nodes,
row sep=.8em, column sep=1.6em, inner sep=1em,
text height=1.5ex, text depth=0.25ex]
{ \cdots &H_p(\T_n;   \widetilde{H}_0(\B(-))) & H_p(\T_n; H_0(\B(-))) & H_p(\T_n; \Z) & \cdots  \\
 \; &\displaystyle  \left\{ \begin{array}{l}  \bigoplus_{ \rank(V)=2}  \St_{n-2}\otimes  \widetilde H_0(\B_2), \; p=n-3 \\ 0, \qquad  \qquad \text{otherwise} \end{array} \right. & E^2_{p,0} & \displaystyle   \left\{ \begin{array}{l} \St_n, \;  p=n-2 \\ \Z,\; \quad  p=0 \\ 0, \; \quad  \text{otherwise}  \end{array} \right.  \\
};
\path[->,font=\scriptsize,>=angle 90]
(m-1-1) edge (m-1-2)
(m-1-2) edge (m-1-3)
(m-1-3) edge (m-1-4)
(m-1-4) edge (m-1-5);
\draw[double]
(m-1-2) -- (m-2-2)
(m-1-3) -- (m-2-3)
(m-1-4) -- (m-2-4);
\end{tikzpicture} \end{center}
Again we conclude that $E^2_{p,0}$ vanishes unless $p$ is $(n-3)$, $(n-2)$, or $0$, which completes the proof of part \ref{E2}. 
\end{proof}

%
%
%
%

\begin{proposition}  \label{StSES} Let $R$ be a PID and let $n \geq 3$. There is an exact sequence: \[0 \m  E^2_{n-2,0} \m \St_n \m H_{n-3}(\T_n;\widetilde H_0(\B(-)) \m 0.\]
\end{proposition}

\begin{proof}
Consider again the short exact sequence of functors
\[ 0\m \widetilde H_0(\B(-)) \m H_0(\B(-)) \m \Z \m 0\]
and the associated long exact sequence on the homology of $\T_n$ described in the proof of \autoref{E2/Einfty}. When $p=(n-2)$, we get the following long exact sequence:

\begin{adjustbox}{width=1.1\textwidth,center}
\begin{tikzpicture} \scriptsize
\matrix(m)[matrix of math nodes,
row sep=.8em, column sep=1.6em, inner sep=1em,
text height=1.5ex, text depth=0.25ex]
{ \;  & H_{n-2}(\T_n;   \widetilde{H}_0(\B(-))) & E^2_{n-2,0} & H_{n-2}(\T_n; \Z) & H_{n-3}(\T_n;  \widetilde{H}_0(\B(-)))  & E^2_{n-3,0} & H_{n-3}(\T_n;\Z) & \; \\
 \; & 0 &  &   \St_n   &  \displaystyle \bigoplus_{\substack{V \subseteq R^n \\ \rank(V)=2}} \St_{n-2} \otimes \widetilde H_0(\B_2) & \left\{\begin{array}{l} 0, \; n >3 \\ \Z, \; n = 3  \end{array}\right. & \left\{\begin{array}{l} 0, \; n >3 \\ \Z, \; n = 3  \end{array}\right. \\ & &\\ 
};
\path[->,font=\scriptsize,>=angle 90]
(m-1-1) edge (m-1-2)
(m-1-2) edge (m-1-3)
(m-1-3) edge (m-1-4)
(m-1-4) edge (m-1-5)
(m-1-5) edge (m-1-6)
(m-1-6) edge (m-1-7)
(m-1-7) edge (m-1-8);
\draw[double]
(m-1-2) -- (m-2-2)
(m-1-4) -- (m-2-4)
(m-1-5) -- (m-2-5)
(m-1-6) -- (m-2-6)
(m-1-7) -- (m-2-7);
\end{tikzpicture} 
\end{adjustbox}

Here, the description of $E^2_{n-3,0}$ follows from \autoref{E2/Einfty} part \ref{E2corner}, and the groups $H_{n-2}(\T_n; \widetilde H_0(\B(-)))$ and $H_{n-2}(\T_n; \widetilde H_0(\B(-)))$ are computed in \autoref{H(T_n,tilde H_0)}. For $n>3$, we obtain the desired short exact sequence immediately. For $n=3$, we note that $E^2_{n-3,0} = E^2_{0,0} \cong E^\infty_{0,0}$ and hence  the map $E^2_{n-3,0} \m H_{n-3}(\T_n; \Z) $ agrees with the map $H_0(\cB_3') \m H_0(\cT_3)$. Since $\cB_3'$ and $\cT_3$ are connected, this map is an isomorphism and so the map $ H_{n-3}(\T_n;  \widetilde{H}_0(\B(-)))  \m  E^2_{n-3,0} $ is the zero map. 
\end{proof}

Since  \[H_{n-3}(\T_n; \widetilde H_0(\B(-))) \cong  \bigoplus_{\substack{V \subseteq R^n \\ \rank(V)=2}} \St_{n-2} \otimes \widetilde H_0(\B_2)\] by  \autoref{H(T_n,tilde H_0)}, the short exact sequence of \autoref{StSES} has the following consequence. 

\begin{corollary}  \label{notsurjcor} Let $n \geq 3$, and let $R$ be a PID such that
  $\B_2(R)$ is not connected.  The map $E^2_{n-2,0} \m \St_n$ is not surjective. 
\end{corollary}

There is an edge morphism $H_{n-2}(\B_n') \m E^\infty_{n-2,0}$. Because there are no differentials into $E^r_{n-2,0}$ for $r>1$, there is a map $E^\infty_{n-2,0} \m E^2_{n-2,0}$.  The following proposition is implicit in the proof of Church--Farb--Putman  \cite[Proof of Theorem A]{CFP-Integrality}. In particular, see Equation 3.1 and the surrounding discussion.

\begin{proposition}  The composition $H_{n-1}(\B_n,\B_n') \m H_{n-2}(\B_n') \m E^\infty_{n-2,0} \m E^2_{n-2,0} \m \St_n$ is the map $\alpha$ described in the beginning of the section. \label{compIntMap}
\end{proposition}

\begin{proposition} Take $n \geq 3$, and let $R$ be PID with $\B_2(R)$ not connected. The composition $H_{n-1}(\B_n,\B_n') \m H_{n-2}(\B_n') \m E^\infty_{n-2,0} \m E^2_{n-2,0} \m \St_n$ is not surjective. \label{compNotSurj}
\end{proposition}

\begin{proof}
The map $E^2_{n-2,0} \m \St_n$ is not surjective so the composition is not surjective.
\end{proof}

\begin{proposition} Let $R$ be a PID with $\B_2(R)$ not connected. The map $\alpha\colon H_{1}(\B_2,\B_2') \m \St_2$ is not surjective.  \label{n=2case}
\end{proposition}

\begin{proof} Since $R$ is PID, $\B_2' \cong \T_2$ so we just need to show the map $\alpha\colon H_{1}(\B_2,\B_2') \m \widetilde H_0(\B_2')$ is not surjective. This map fits into an exact sequence: \[H_{1}(\B_2,\B_2') \m \widetilde H_0(\B_2') \m \widetilde H_0(\B_2) \m H_0(\B_2,\B'_2).\] The relative homology group $H_0(\B_2,\B'_2)$ vanishes because $\B'_2$ is the $0$-skeleton of $\B_2$. Since $\widetilde H_0(\B_2)$ is not zero, $\alpha\colon H_{1}(\B_2,\B_2') \m \widetilde H_0(\B_2') $ is not surjective. 
\end{proof}

We now prove \autoref{main}. 

\begin{proof}[Proof of \autoref{main}]
Let $\cO_K$ be a quadratic imaginary number ring which is a PID but not Euclidean. By \autoref{B2notCon}, $\B_2$ is not connected. 

By \autoref{compIntMap} and \autoref{compNotSurj} for $n \geq 3$ and by \autoref{n=2case} for $n=2$, the map \[\alpha\colon H_{n-1}(\B_n,\B_n') \m \St_n\] is not surjective.
The image of $\alpha$ is the submodule of the Steinberg module generated by integral apartment classes. Thus, $\St_n(K)$ is not generated by integral apartment classes. \end{proof}

\begin{remark}
The arguments show the Steinberg module of any PID is not generated by integral apartment classes if $\B_2$ not connected. See Cohn \cite[Theorem C]{Cohn} and Church--Farb--Putman  \cite[Proof of Proposition 2.1]{CFP-Integrality} for examples of rings of integers in function fields with $\B_2$ not connected. 

\end{remark}

\begin{remark}
In this section, we used the map of posets spectral sequence to prove a certain map is not surjective using that a particular simplicial complex is not connected. This is an adaptation of the arguments of \cite{MPP} where a conjecture of Lee--Szczarba \cite[Page 28]{LS} is disproved. There, the fact that a certain simplicial complex is not simply connected is used to show a map is not injective. 
\end{remark}

\section{Non-vanishing of top degree cohomology}

In this section, we show that our proof of non-integrality can sometimes be adapted to show non-vanishing of the cohomology in the virtual cohomological dimension. Throughout, $d$ will denote a negative squarefree integer. 


\subsection{An equivariant calculation of $H^2(\SL_2(\cO_d);\Q)$ for
  $d=-43$, $-67$, $-163$ 
}

\hfill

We begin by recalling a calculation of $H^2(\SL_2(\cO_d);\Q)$ for $d=-43,-67,-163$ and then describe our calculation of $H^2(\GL_2(\cO_d);\Q)$ for these rings. We also compute the torsion at primes greater than three. Note that $\nu_2=2$ so this is the virtual cohomological dimension. Since the units in these rings are $\{-1,1\}$, $H^i(\SL_n(\cO_d))$ is naturally an $\Z/2\Z$--representation. Knowing both $H^i(\SL_n(\cO_d); \Q)$ and $H^i(\GL_n(\cO_d);\Q)$ allows one to compute $H^i(\SL_n(\cO_d);\Q)$ as a $\Z/2\Z$--representation. 

The following was proven by Rahm \cite[Proposition 1]{Rahm}, with some cases previously known by the work of Vogtmann \cite{Vo}. Rahm's result concerns the integer homology of the group $\PSL_2(\cO_d)$, whose rational homology agrees with that of $\SL_2(\cO_d)$.

\begin{theorem}[{\cite[Proposition 1]{Rahm}}] \label{NonvanishingImaginary} 
Let $\cO_d$ denote the ring of integers in the quadratic number field $\Q(\sqrt{d})$. Then \[\dim_{\Q} H^{2}(\SL_{2}(\cO_{d});\Q) \geq 
 \begin{cases}
1 & \text{ for $d=-43$}, \\
2 & \text{ for $d=-67$}, \\
6 & \text{ for $d=-163$}.
\end{cases}
\]

\end{theorem}

We now describe the analogous calculation for $\GL_2(\OO_d)$. This will follow from the methods of \cite{aimpaper}, \cite[\S3]{PerfFormModGrp},
\cite[\S2]{soule-3torsion}, and \cite{agm4}.
For any positive integer $b$, let $\cS_{b}$ be the Serre class of finite
abelian groups with orders only divisible by primes less than or equal
to $b$ \cite{serre}.  Let $\Gamma$ be a finite-index subgroup in
$\GL_n(\OO_K)$. If $b$ is larger than all the primes dividing the orders of
finite subgroups of $\Gamma$, then modulo $\cS_{b}$ the group cohomology of  
$\Gamma$ can be computed using the homology of the Voronoi complex  $\Vor_{n,d}$, a complex coming from a decomposition of a certain space of Hermitian forms.  In particular, up to
torsion divisible by primes less than or equal to $b$, the Voronoi complex
captures the group cohomology.  The number $b$ can be explicitly bounded from
above.  We refer the 
reader to \cite[\S3.1]{PerfFormModGrp} and \cite[\S3]{aimpaper} for the precise
definition of $\Vor_{n,d}$.  We give a brief description in the proof of \autoref{thm5.5}.

\begin{proposition}[{\cite[Lemma~3.9]{aimpaper}}] \label{prop:bound_b}
  Let $p$ be an odd prime, and let $K$ be an imaginary quadratic
  field.  If $g \in \GL_n(K)$ has order $p$, then 
  \[p \leq 
  \begin{cases}
    n + 1 & \text{if $p \equiv 1 \bmod{4}$,}\\
    2n + 1 & \text{otherwise.}
  \end{cases}
  \]
\end{proposition}

\begin{theorem}[{\cite[Theorem~3.7]{aimpaper}}]\label{thm:aimpaper}
Let $b$ be an upper bound on the torsion primes for $\GL_n(\OO_d)$.  Modulo the
Serre class $\cS_b$, 
\[H_i(\Vor_{n,d}) \cong H_{i - (n - 1)}(\GL_n(\OO_d);\St_n(\Q(\sqrt d))  ) 
\cong H^{n^2 - 1 - i}(\GL_n(\OO_d)).\]
\end{theorem}

By \autoref{prop:bound_b}, the torsion primes for $\GL_2(\OO_d)$ are $2$ and
$3$.  
\begin{corollary}
Modulo $\cS_3$, 
\[H_1(\Vor_{2,d}) \cong H_0(\GL_2(\OO_d);\St_n(\Q(\sqrt d))) \cong H^2(\GL_2(\OO_d)).\]  
\end{corollary}


\begin{theorem}\label{thm5.5}
\begin{gather*}
    H_1(\Vor_{2,-43}) \cong \ZZ/2\ZZ, \quad
    H_1(\Vor_{2,-67}) \cong (\ZZ/2\ZZ)^2, \quad
    H_1(\Vor_{2,-163}) \cong (\ZZ/2\ZZ)^6,\\
    H_2(\Vor_{2,-43}) \cong \ZZ/2\ZZ, \quad
    H_2(\Vor_{2,-67}) \cong \ZZ/2\ZZ, \quad
    H_2(\Vor_{2,-163}) \cong (\ZZ/2\ZZ)^2.
\end{gather*}
\end{theorem}
\begin{proof}
The algorithms for computing Voronoi homology are given in
\cite[\S6]{aimpaper}.  The three stages are: 
\begin{enumerate}
\item determining the perfect forms;
\item computing the Voronoi complex and the differentials;
\item computing the homology.
\end{enumerate}
We implemented these steps using \textsc{magma}
\cite{magma} and lrs Vertex Enumeration/Convex Hull package \cite{Avis2000}, a C
implementation of the reverse search algorithm for  vertex enumeration/convex
hull problems.   We include some details to give a sense of the computational
task.   We remark the determination of perfect forms is already known
\cite{bianchi}, and the current computational results are consistent with the
earlier results.  

Let $\cH^2(\CC)$ denote the $4$-dimensional real vector space of $2
\times 2$ Hermitian matrices with complex coefficients.  Using the chosen complex embedding of $K=\Q(\sqrt d)$ we can view $\cH^2(K)$, the Hermitian
matrices with coefficients in $K$, as a subset of $\cH^2(\CC)$.
Moreover, this embedding allows us to view $\cH^2(\CC)$ as a
$\QQ$-vector space such that the rational points of $\cH^2(\CC)$ are
exactly $\cH^2(K)$.  Let $C^* \subset \cH^2(\CC)$ denote the non-zero positive
semi-definite Hermitian forms with $K$-rational kernel, and let $X$ denote the
quotient of $C^*$ by positive homothety.  There is a natural identification of a
subset of $X^*$ with hyperbolic $3$-space $\HH^3$.  Voronoi theory describes a
decomposition of $X^*$ in terms of configurations of minimal vectors of
Hermitian forms, which gives rise to a tessellation of $\HH^3$ by ideal
$3$-dimensional hyperbolic polytopes.  These polytopes with certain gluing maps
determine the Voronoi complex and differentials. The Voronoi complex is obtained by working relative to the boundary, in this case exactly the $0$-cells.

The computations for $d = -67$ and $d= -163$ are larger, so we just summarize
some of the key features after giving details for the case $d = -43$.

Let $\omega = \frac{1 + \sqrt{-43}}{2}$.
Consider the vectors $v_1, v_2, \dots, v_{21}$:
\begin{multline*}
\mat{-3\omega+3\\2\omega-12} , \mat{-\omega+3\\-5} , \mat{3\\-\omega-2} ,
\mat{2\omega+7\\-4\omega+2} , \mat{-\omega+10\\-2\omega-10} ,
\mat{\omega\\-\omega+3} , \mat{1\\-1}, \\
\mat{\omega+1\\-\omega+2} ,
\mat{-\omega+4\\-5} , \mat{4\\-\omega-3} , \mat{\omega\\-\omega+2} ,
\mat{\omega+1\\-\omega+3} , \mat{3\\-\omega-1} , \mat{4\\-\omega-2} , \\
\mat{0\\1}
, \mat{1\\0} , \mat{-\omega+2\\-4} , \mat{-\omega+2\\-3} ,
\mat{\omega+2\\-\omega+2} , \mat{-\omega+3\\-4} , \mat{4\\-\omega-1}.
\end{multline*}
Using the Voronoi algorithm adapted to this 
case, we find that there are four equivalence classes of perfect forms.  We
describe a perfect form by its set of minimal vectors (up to $\pm 1$) by giving
the indices of the vectors that are the minimal vectors for that form.  For
example, $\{1,2,5\}$ represents a form with minimal vectors  $\{\pm v_1,
  \pm v_2, \pm v_5\}$.  
We find explicit representatives for each class of perfect forms:
\begin{align*}
\phi_1 &= \{ 1, 2, 3, 4, 5, 6 \},\\
\phi_2 &= \{ 6, 7, 8, 9, 10, 11 \},\\
\phi_3 &= \{ 2, 3, 6, 7, 8, 12, 13, 14, 15 \},\\
\phi_4 &= \{ 7, 8, 12, 13, 14, 15, 16, 17, 18, 19, 20, 21 \}.
\end{align*}
These perfect forms determine hyperbolic polytopes.  The facets of $\phi_1$ are 
\[
\{ 1, 3, 5, 6 \},
\{ 2, 4, 5, 6 \},
\{ 1, 2, 3, 4 \},
\{ 2, 3, 6 \},
\{ 1, 4, 5 \}.
\]
The facets of $\phi_2$ are
\[
\{ 3, 4, 5 \},
\{ 2, 3, 4, 6 \},
\{ 1, 4, 5, 6 \},
\{ 1, 2, 3, 5 \},
\{ 1, 2, 6 \}.
\]
The facets of $\phi_3$ are
\[
\{ 1, 2, 3, 6, 7, 8 \},
\{ 1, 2, 9 \},
\{ 5, 6, 8 \},
\{ 1, 5, 8, 9 \},
\{ 2, 4, 7, 9 \},
\{ 4, 5, 9 \},
\{ 3, 4, 5, 6 \},
\{ 3, 4, 7 \}.
\]
The facets of $\phi_4$ are 
\begin{multline*}
\{ 1, 3, 4, 9, 10, 12 \},
\{ 2, 5, 6, 9, 11, 12 \},
\{ 4, 5, 7, 8, 10, 11 \},
\{ 6, 7, 11 \},\\
\{ 1, 2, 3, 6, 7, 8 \},
\{ 1, 2, 12 \},
\{ 3, 8, 10 \},
\{ 4, 5, 9 \}.
\end{multline*}
We see that $\phi_1$ and $\phi_2$ give rise to triangular prisms, $\phi_3$ gives
rise to a hexagonal cap, and $\phi_4$ gives rise to a truncated tetrahedron.
\begin{center}
\begin{tabular}{cccc}
  \includegraphics[width=0.15\textwidth]{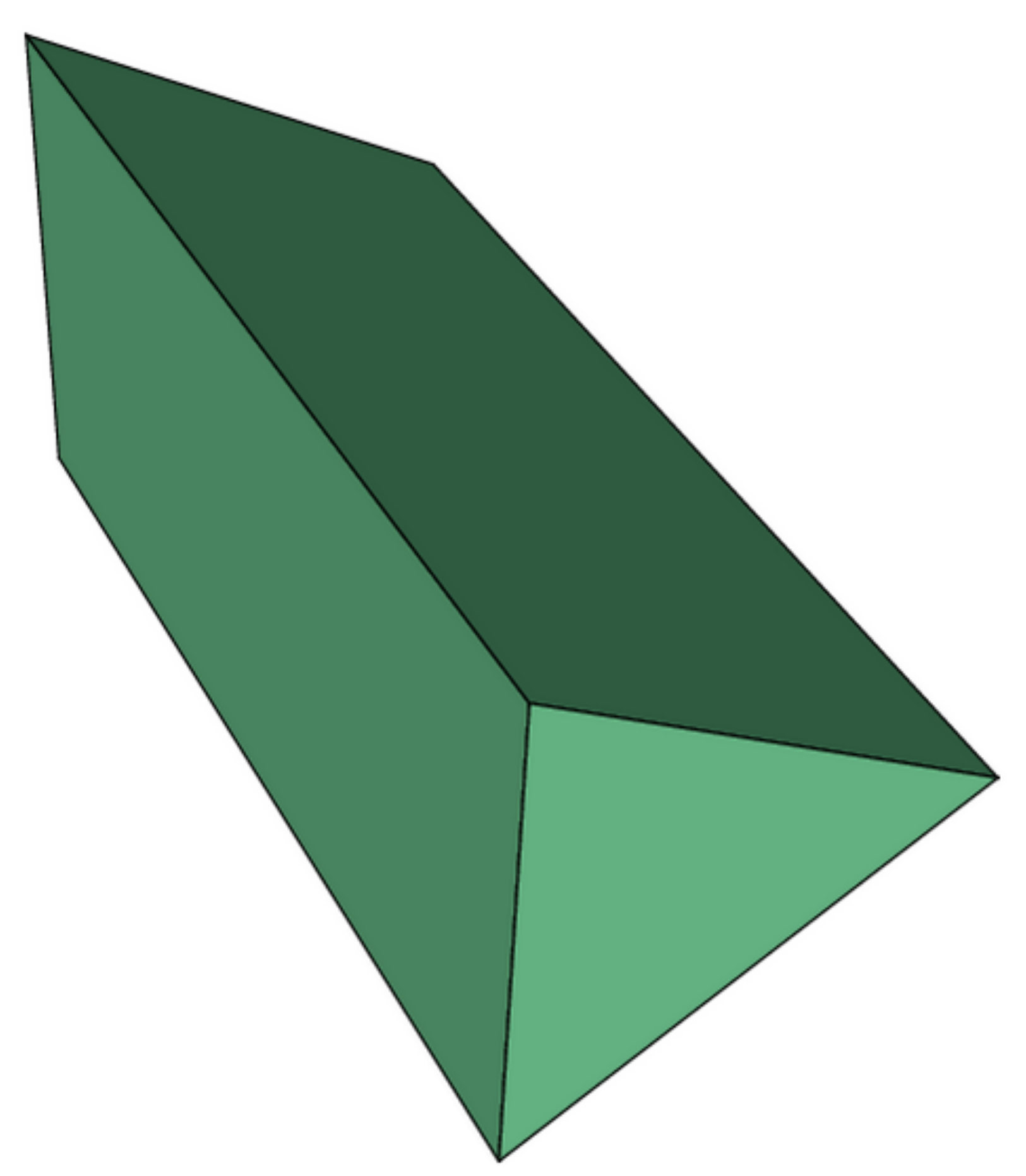} &
  \includegraphics[width=0.15\textwidth]{triangularprism.pdf} & 
  \includegraphics[width=0.15\textwidth]{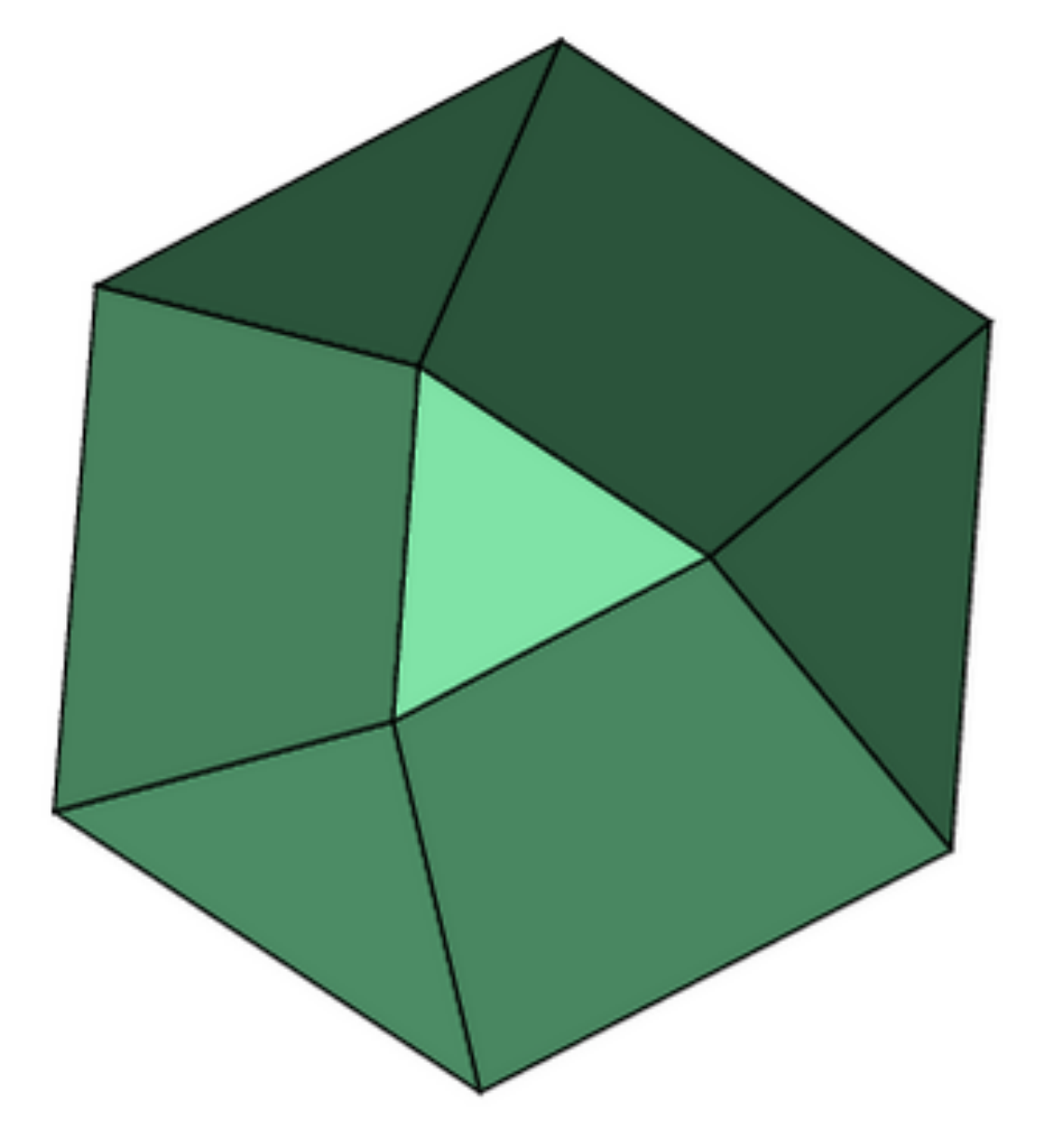} & 
 \includegraphics[width=0.15\textwidth]{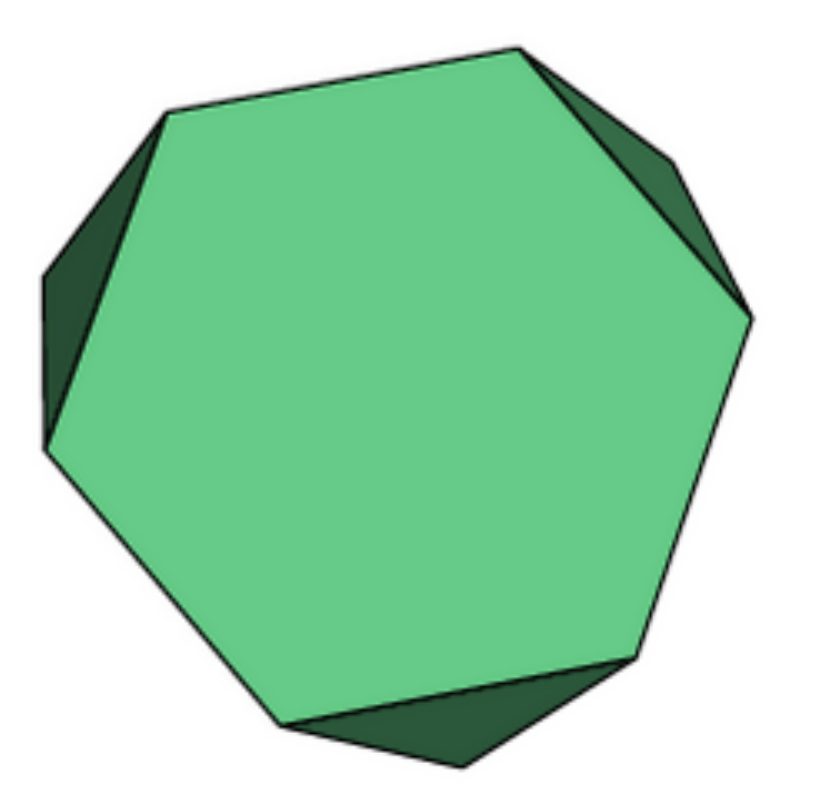}\\
Triangular Prism & Triangular Prism & Hexagonal Cap & Truncated Tetrahedron\\
$\phi_1$ & $\phi_2$ & $\phi_3$ & $\phi_4$
\end{tabular}
\end{center}
The polytopes are given explicitly, and the gluing maps can be computed.  There
are four types of $3$-dimensional cells, six types of $2$-dimensional faces, and
four types of $1$-dimensional edges.  We have $H_1(\Vor_{2,-43})$ is the
cokernel of the 
differential from $2$-cells to $1$-cells, and $H_2(\Vor_{2,-43})$ is the kernel
of the differential from $2$-cells to $1$-cells modulo the image of 
the differential from $3$-cells to $2$-cells.  An explicit linear algebra
computation gives the result. 

For $d = -67$, there are seven equivalence classes of perfect forms that give rise to one octahedron, two triangular prisms, one hexagonal cap, two square pyramids, and one truncated tetrahedron.  This gives seven types of $3$-dimensional cells, thirteen types of $2$-dimensional faces, and eight
types of $1$-dimensional edges.  Again, we compute the differentials, and an explicit linear algebra computation gives the result.

For $d = -163$, there are twenty-five equivalence classes of perfect forms that
give rise to eleven tetrahedra, one cuboctahedron, eight triangular prisms, two
hexagonal caps, and three square pyramids.  This gives twenty-five types of
$3$-dimensional cells, forty-nine types of $2$-dimensional faces, and
twenty-seven types of $1$-dimensional edges.  Again, we compute the
differentials, and an explicit linear algebra computation gives the result.
\end{proof}
\setcounter{MaxMatrixCols}{21}
\begin{corollary} \label{GL2}
For $d \in \{-43, -67, -163\}$, modulo $\cS_3$, 
\[H_0(\GL_2(\OO_d);\St_n(\QQ(\sqrt{d}))) \cong H^2(\GL_2(\OO_d)) \cong 0.\]  
\end{corollary}

\subsection{Non-vanishing for $2n \geq 4$}

\hfill

The goal of this subsection is to leverage the calculations of the previous section to prove \autoref{mainCohomology}.

 \begin{lemma} \label{StB2}
 Let $R$ be a PID. Then  $\widetilde H_0(\cB_2;\Q)_{\SL_2(R)} \cong (\St_2 \otimes \Q)_{\SL_2(R)}$. 
 \end{lemma}

\begin{proof}

We will first show that $H_1(\cB_2,\cB_2';\Q)_{\SL_2(R)} \cong 0$ by the usual argument showing integrality implies homological vanishing (see \cite[proof of Theorem C]{CFP-Integrality}). As discussed in \autoref{sec4}, $H_1(\cB_2,\cB_2')$ is the quotient of the free abelian group on symbols $(F_1,F_2)$ with $F_1$ and $F_2$ rank-one free submodules of $R^2$ with $R^2 = F_1 \oplus F_2$ modulo the relation that $(F_1,F_2)=-(F_2,F_1)$. Let $g$ be an element of $\SL_2(R)$ with $g(F_1)=F_2$ and $g(F_2)=F_1$. We have that $g$ acts via multiplication by $-1$ on $(F_1,F_2)$. Thus $(F_1,F_2)=-(F_1,F_2)$ in $H_1(\cB_2,\cB_2';\Q)_{\SL_2(R)}$. Since $2$ is invertible in $\Q$, this element and hence the group vanishes. 

Since $\cB_2'$ is the zero skeleton of $\cB_2$, $H_0(\cB_2,\cB_2';\Q) \cong 0$. Since $R$ is a PID, $\cB_2' \cong \cT_2$. Thus, $\widetilde H_0(\cB_2') \cong \St_2$. The claim now follows from applying the coinvariants functor to the exact sequence: \[ H_1(\cB_2,\cB_2';\Q)  \m \widetilde H_0(\cB_2';\Q) \m \widetilde H_0(\cB_2;\Q) \m H_0(\cB_2,\cB_2';\Q). \qedhere  \] 
\end{proof}


Since $\St_n(R)$ is a $\GL_n(R)$--representation, $\St_n(R)_{\SL_n(R)}$ inherits the structure of a $\GL_n(R)/\SL_n(R)$--representation. Similarly, the $\SL_n(R)$--coinvariants of \[\bigoplus_{\substack{V \subseteq R^n \\ \rank(V)=2}}  \St(R^n/V) \otimes  \widetilde H_0(\B(V))\] has a natural linear action of $\GL_n(R)/\SL_n(R)$. Combining results from the previous section gives the following.

\begin{lemma} \label{lemmaSurj} Let $R$ be a PID with group of units $\{\pm 1\}$.  There is a surjection  \[\St_n(R)_{\SL_n(R)} \longrightarrow \left( \bigoplus_{\substack{V \subseteq R^n \\ \rank(V)=2}}  \St(R^n/V) \otimes  \widetilde H_0(\B(V))\right)_{\SL_n(R)}.\]  
This surjection is equivariant with respect to the action of $\Z/2\Z \cong \GL_n(R)/\SL_n(R)$.
\end{lemma}
\begin{proof} 

 By \autoref{H(T_n,tilde H_0)} and \autoref{StSES}, there is a $\GL_n(R)$-equivariant surjection 
\[ \St_n(R) \longrightarrow  H_{n-3}(\T_n;\widetilde H_0(\B(-)))  \cong \bigoplus_{\substack{V \subseteq R^n \\ \rank(V)=2}}  \St(R^n/V) \otimes  \widetilde H_0(\B(V)). \]  The claim follows from the right-exactness of coinvariants. 
\end{proof}
 
Both $(\St_{n-2}(R))_{\SL_{n-2}(R)}$ and $(\widetilde H_0(\B_2))_{\SL_2(R)}$ have a linear action by $R^\times \cong \GL_{n-2}(R)/\SL_{n-2}(R) \cong \GL_2(R)/\SL_2(R)$. Because $R^\times$ is commutative,  $(\St_{n-2}(R))_{\SL_{n-2}(R)} \otimes_{\Z[R^\times]} (\widetilde H_0(\B_2))_{\SL_2(R)}$ has a linear action by $R^\times$ (that only acts on one factor) as well.

\begin{lemma} Let $R$ be a PID with group of units $\{\pm 1\}$. There is an isomorphism
\[ \left( \bigoplus_{\substack{V \subseteq R^n \\ \rank(V)=2}}  \St(R^n/V) \otimes  \widetilde H_0(\B(V))\right)_{\SL_n(R)} \cong  (\St_{n-2}(R))_{\SL_{n-2}(R)} \otimes_{\Z[\Z/2\Z]} (\widetilde H_0(\B_2))_{\SL_2(R)},  \]  and this isomorphism is equivariant with respect to the action of $\Z/2\Z \cong \GL_n(R)/\SL_n(R) \cong R^\times$.
\label{lemmaCoinvariance}

\end{lemma}

\begin{proof}  
Define
\[ G=  \left\{\begin{bmatrix} A & * \\ 0 & B \end{bmatrix} \; \middle| \; A \in \mathrm{GL}_2(R),  B \in \mathrm{GL}_{n-2}(R), \; \mathrm{det}(A)\mathrm{det}(B)=1 \right\} \subseteq \SL_n(R) \] to be the stabilizer of the standard copy of $R^2$ in $R^n$. Then \begin{align*}
& \left( \bigoplus_{\substack{V \subseteq R^n \\ \rank(V)=2}}  \St(R^n/V) \otimes_{\Z}  \widetilde H_0(\B(V))\right)_{\SL_n(R)} \\
 &\cong \Z \otimes_{\Z[\SL_n(R)]} \left( \bigoplus_{\substack{V \subseteq R^n \\ \rank(V)=2}}  \St(R^n/V) \otimes_{\Z}   \widetilde H_0(\B(V))\right)\\ 
& \cong \Z \otimes_{\Z[\SL_n(R)]} \left( \Z[\SL_n(R)] \otimes_G \left( \St(R^n/R^2) \otimes_{\Z}   \widetilde H_0(\B(R^2)) \right) \right)\\ 
&\cong \Z  \otimes_G \left( \St (R^n/R^2) \otimes_{\Z}   \widetilde H_0(\B(R^2)) \right). 
\end{align*}
Observe that the subgroup $\left[\begin{smallmatrix} \mathrm{Id}_2 & * \\ 0 & \mathrm{Id}_{n-2} \end{smallmatrix}\right] \subseteq G$ acts trivially, so the action by $G$ factors through an action of 
 \[H =  \left\{\begin{bmatrix} A & 0 \\ 0 & B \end{bmatrix} \; \middle| \; A \in \mathrm{GL}_2(R),  B \in \mathrm{GL}_{n-2}(R), \; \mathrm{det}(A)\mathrm{det}(B)=1 \right\} \quad \cong \quad \Big( \mathrm{SL}_2(R) \times \mathrm{SL}_{n-2}(R) \Big) \rtimes  \mathbb{Z}/2\mathbb{Z}.\] Thus 
\begin{align*}
&\;\;\Z  \otimes_G \left( \St (R^n/R^2) \otimes_{\Z}   \widetilde H_0(\B(R^2)) \right)  \\ 
&\cong   \left( \St (R^n/R^2) \otimes_{\Z}   \widetilde H_0(\B(R^2)) \right)_H \\
&\cong \left( \left(\St_{n-2}(R^n/R^2) \otimes_{\Z}  \widetilde H_0(\B(R^2))\right)_{\mathrm{SL}_2(R) \times \mathrm{SL}_{n-2}(R)} \right)_{H / (\mathrm{SL}_2(R) \times \mathrm{SL}_{n-2}(R))} \\
&\cong \left(\St_{n-2}(R^n/R^2)_{\mathrm{SL}_{n-2}(R)} \otimes_{\Z}  \widetilde H_0(\B(R^2))_{\mathrm{SL}_{2}(R)}\right)_{\Z/2\Z} \\ 
&\cong \St_{n-2}(R^n/R^2)_{\mathrm{SL}_{n-2}(R)} \otimes_{\Z[\Z/2\Z]}   \widetilde H_0(\B(R^2))_{\mathrm{SL}_{2}(R)} \\ 
\end{align*}
as claimed. 
\end{proof}


 By Borel-Serre duality, the following is equivalent to \autoref{mainCohomology}.

\begin{proposition} For all $n\ge 1$, we have \[\dim_{\Q} H_0(\SL_{2n}(\cO_{d});\Q \otimes \St_{2n}(\Q(\sqrt d))) \geq 
 \begin{cases}
1 & \text{ for $d=-43$,} \\
2^n & \text{ for $d=-67$,} \\
6^n & \text{ for $d=-163$.}
\end{cases}
\]
\end{proposition}

\begin{proof}
Recall that the coinvariants $\Q\otimes_{\Z}(\St_{n}(\Q(\sqrt d)))_{\SL_{n}(\OO_d)}$ are a representation of $\Z/2\Z \cong \GL_n(\OO_d)/\SL_n(\OO_d)$. Let $t_n$ denote the multiplicity of the trivial representation, and $s_n$ denote the multiplicity of the sign representation in $\Q\otimes_{\Z}(\St_{n}(\Q(\sqrt d)))_{\SL_{n}(\OO_d)}$. By \autoref{StB2}, \autoref{lemmaSurj} and \autoref{lemmaCoinvariance}, there is an equivariant surjection \[ (\St_n)_{\SL_n} \otimes \Q \twoheadrightarrow  \left( (\St_{n-2})_{\SL_{n-2}} \otimes_{\Z[\Z/2\Z]} (\St_2)_{\SL_2}  \right) \otimes \Q \] and so \[ t_n  \geq t_{n-2}t_2  \qquad \text{and} \qquad  s_n \geq s_{n-2}s_2.  \] Thus, 
\[  \dim_{\Q} H_0(\SL_{2n}(\cO_d);\Q \otimes \St_{2n}(\Q(\sqrt d)))=  t_{2n}+s_{2n}  \geq (t_2)^n+(s_2)^n. \] Since \[t_n=\dim_\Q H_0(\GL_n(\OO_d);\Q \otimes \St_n(\Q(\sqrt d))) \text{ and }t_n+s_n=\dim_\Q H_0(\SL_n(\OO_d);\Q \otimes \St_n(\Q(\sqrt d))),\]  \autoref{NonvanishingImaginary} and \autoref{GL2} give $t_2=0$ and \[s_2=
 \begin{cases}
1 & \text{ for $d=-43$,} \\
2 & \text{ for $d=-67$,} \\
6 & \text{ for $d=-163$.}
\end{cases} \qedhere
\] \end{proof}

\begin{remark} 
Since $s_1=0$, one cannot easily use the proof strategy of \autoref{mainCohomology} to show non-vanishing of $H^{\nu_n}(\SL_n(\cO_d);\Q)$ for all $n$ and $d=-43,-67,$ or $-163$. In fact, $s_{n}=0$ for all $n$ odd. To see this, note that for $n$ odd, the $\Z/2\Z$-action on $H^{\nu_n}(\SL_n(\cO_d);\Q)$ is induced by the action on the Tits building given by multiplication by $\pm \Id_n$ which is trivial at the level of posets. 
\end{remark}

\begin{remark} 
Let $d=-19, -43,-67,$ or $-163$, let $\mathfrak p$ be a prime ideal in $\cO_d$, and let $R$ be $\cO_d$ with $\mathfrak p$ inverted. Then it follows from Church--Farb--Putman \cite[Theorem A]{CFP-Integrality} that $\St_n(\Q(\sqrt d))$ is generated by integral (with respect to the ring $R$) apartment classes. The proof of \cite[Theorem C]{CFP-Integrality} shows that $H_0(\SL_n(R);\St_n(\Q(\sqrt d)) \otimes \Z[\frac{1}{2}]) \cong 0$ for all $n \geq 2$. In other words, the phenomenon explored in this paper do not persist after inverting even a single prime.

\end{remark}

\begin{remark} \label{mapBounds}
Let $d=-43$, $-67$, or $-163$. The bounds in \autoref{mainCohomology} come from a surjection \[H_0(\SL_{2n}(\OO_d);\St_n(\Q(\sqrt d)) \otimes \Q ) \twoheadrightarrow  \left( H_0(\SL_2(\OO_d);\St_2(\Q(\sqrt d));\Q ) \right)^{n \otimes_{\Z/2\Z}}. \] It follows from our proofs that this map is induced by taking coinvariants of a (not necessarily surjective) map \[\Delta\colon \St_{2n}(\Q(\sqrt d))  \m 
\bigoplus_{0 \subsetneq V_1 \subsetneq V_2  \subsetneq \cdots  \subsetneq V_{n-1} \subsetneq \left(\Q(\sqrt d)\right)^{2n} ,\, \dim V_i \text{ even}} \left (\bigotimes_i \St(V_{i+1}/V_i) \right). \]  Since (not necessarily integral) apartment classes generate the Steinberg module, it suffices to describe the map $\Delta$ on apartment classes. Fix lines $L_1,\ldots,L_{2n}$ with $\left(\Q(\sqrt d) \right)^{2n} = L_1 \oplus \ldots \oplus L_{2n}$ and let $[L_1,\ldots, L_{2n}] \in \St_{2n}(\Q(\sqrt d))$ denote the associated apartment class. Let $X_n$ be the set of permutations of $2n$ such that $\sigma(2i-1)<\sigma(2i)$ for all $i$. After unpacking the definition of the connecting homomorphism used in our proof, one can check that $\Delta$ is given by the formula: \[ \Delta([L_1,\ldots, L_{2n}])=\sum_{\sigma \in X_n} \sgn(\sigma) [L_{\sigma(1)},L_{\sigma(2)}] \otimes  \cdots \otimes [L_{\sigma(2n-1)},{L_{\sigma(2n)}}]  \] with $[L_{\sigma(2i-1)},L_{\sigma(2i)}] \in \St\left((L_{\sigma(1)} \oplus \cdots \oplus L_{\sigma(2i)} )/ (L_{\sigma(1)} \oplus \cdots \oplus  L_{\sigma(2i-2)} )\right) $.

\end{remark}

{\footnotesize
\bibliographystyle{amsalpha}
\bibliography{Steinberg}

\newcommand{\etalchar}[1]{$^{#1}$}
\providecommand{\bysame}{\leavevmode\hbox to3em{\hrulefill}\thinspace}
\providecommand{\MR}{\relax\ifhmode\unskip\space\fi MR }
\providecommand{\MRhref}[2]{%
  \href{http://www.ams.org/mathscinet-getitem?mr=#1}{#2}
}
\providecommand{\href}[2]{#2}
\begin{thebibliography}{DSGG{\etalchar{+}}16}

\bibitem[AGM11]{agm4}
Avner Ash, Paul~E. Gunnells, and Mark McConnell, \emph{Torsion in the
  cohomology of congruence subgroups of {${\rm SL}(4,\Bbb Z)$} and {G}alois
  representations}, J. Algebra \textbf{325} (2011), 404--415. \MR{2745546
  (2012b:11084)}

\bibitem[AR79]{AR}
Avner Ash and Lee Rudolph, \emph{The modular symbol and continued fractions in
  higher dimensions}, Invent. Math. \textbf{55} (1979), no.~3, 241--250.
  \MR{553998}

\bibitem[Avi00]{Avis2000}
David Avis, \emph{A revised implementation of the reverse search vertex
  enumeration algorithm}, Polytopes---combinatorics and computation
  ({O}berwolfach, 1997), DMV Sem., vol.~29, Birkh\"{a}user, Basel, 2000,
  pp.~177--198. \MR{1785299}

\bibitem[BCP97]{magma}
Wieb Bosma, John Cannon, and Catherine Playoust, \emph{The {M}agma algebra
  system. {I}. {T}he user language}, J. Symbolic Comput. \textbf{24} (1997),
  no.~3-4, 235--265, Computational algebra and number theory (London, 1993).
  \MR{MR1484478}

\bibitem[BS73]{BoSe}
Amrand Borel and Jean-Pierre Serre, \emph{Corners and arithmetic groups},
  Comment. Math. Helv. \textbf{48} (1973), 436--491, Avec un appendice:
  Arrondissement des vari\'et\'es \`a coins, par A. Douady et L. H\'erault.
  \MR{0387495}

\bibitem[CFP15]{CFP-Integrality}
Thomas Church, Benson Farb, and Andrew Putman, \emph{Integrality in the
  {S}teinberg module and the top-dimensional cohomology of
  {SL$_n(\mathcal{O}_K)$}}, Preprint, to appear in Amer. J. Math. (2015),
  \url{http://arXiv:1501.01307}.

\bibitem[Cha87]{Charney-Generalization}
Ruth Charney, \emph{A generalization of a theorem of {V}ogtmann}, Proceedings
  of the {N}orthwestern conference on cohomology of groups ({E}vanston, {I}ll.,
  1985), vol.~44, 1987, pp.~107--125. \MR{885099}

\bibitem[Coh66]{Cohn}
Paul~M. Cohn, \emph{On the structure of the {$\mathrm{GL}_2$} of a ring},
  Publications Math{\'e}matiques de l'Institut des Hautes {\'E}tudes
  Scientifiques \textbf{30} (1966), no.~1, 5--53.

\bibitem[CP17]{CP}
Thomas Church and Andrew Putman, \emph{The codimension-one cohomology of {${\rm
  SL}_n\Bbb Z$}}, Geom. Topol. \textbf{21} (2017), no.~2, 999--1032.
  \MR{3626596}

\bibitem[DSGG{\etalchar{+}}16]{aimpaper}
Mathieu Dutour~Sikiri\'c, Herbert Gangl, Paul~E. Gunnells, Jonathan Hanke,
  Achill Sch\"urmann, and Dan Yasaki, \emph{On the cohomology of linear groups
  over imaginary quadratic fields}, J. Pure Appl. Algebra \textbf{220} (2016),
  no.~7, 2564--2589. \MR{3457984}

\bibitem[EVGS13]{PerfFormModGrp}
Philippe Elbaz-Vincent, Herbert Gangl, and Christophe Soul{\'e}, \emph{Perfect
  forms, {K}-theory and the cohomology of modular groups}, Adv. Math.
  \textbf{245} (2013), 587--624. \MR{3084439}

\bibitem[Gar73]{Garland}
Howard Garland, \emph{{$p$}-adic curvature and the cohomology of discrete
  subgroups of {$p$}-adic groups}, Ann. of Math. (2) \textbf{97} (1973),
  375--423. \MR{0320180}

\bibitem[HW10]{hatcherwahl}
Allen Hatcher and Nathalie Wahl, \emph{Stabilization for mapping class groups
  of 3-manifolds}, Duke Math. J. \textbf{155} (2010), no.~2, 205--269.
  \MR{2736166 (2012c:57001)}

\bibitem[LS76]{LS}
Ronnie Lee and R.~H. Szczarba, \emph{On the homology and cohomology of
  congruence subgroups}, Invent. Math. \textbf{33} (1976), no.~1, 15--53.
  \MR{0422498}

\bibitem[MPP]{MPP}
Jeremy Miller, Peter Patzt, and Andrew Putman, \emph{On the top dimensional
  cohomology groups of congruence subgroups of {$SL_n(\mathbb Z$)}}, Preprint,
  \url{https://arxiv.org/abs/1909.02661}.

\bibitem[Qui73]{Quillen-Ki}
Daniel Quillen, \emph{Finite generation of the groups {$K_i$} of rings of
  algebraic integers}, Higher K-Theories, Springer, 1973, pp.~179--198.

\bibitem[Qui78]{Quillen-Poset}
\bysame, \emph{Homotopy properties of the poset of nontrivial p-subgroups of a
  group}, Advances in Mathematics \textbf{28} (1978), no.~2, 101--128.

\bibitem[Rah13]{Rahm}
Alexander Rahm, \emph{The homological torsion of ${PSL}_2$ of the imaginary
  quadratic integers}, Transactions of the American Mathematical Society
  \textbf{365} (2013), no.~3, 1603--1635.

\bibitem[Ser53]{serre}
Jean-Pierre Serre, \emph{Groupes d'homotopie et classes de groupes ab\'eliens},
  Ann. of Math. (2) \textbf{58} (1953), 258--294. \MR{0059548 (15,548c)}

\bibitem[Sol69]{Solomon}
Louis Solomon, \emph{The {S}teinberg character of a finite group with
  {$BN$}-pair}, Theory of {F}inite {G}roups ({S}ymposium, {H}arvard {U}niv.,
  {C}ambridge, {M}ass., 1968), Benjamin, New York, 1969, pp.~213--221.
  \MR{0246951}

\bibitem[Sou00]{soule-3torsion}
Christophe Soul{\'e}, \emph{On the {$3$}-torsion in {$K_4({\bf Z})$}}, Topology
  \textbf{39} (2000), no.~2, 259--265. \MR{1722032 (2000i:19008)}

\bibitem[vdK80]{vdK}
Wilberd van~der Kallen, \emph{Homology stability for linear groups}, Invent.
  Math. \textbf{60} (1980), no.~3, 269--295. \MR{586429 (82c:18011)}

\bibitem[Vog85]{Vo}
Karen Vogtmann, \emph{Rational homology of {B}ianchi groups}, Math. Ann.
  \textbf{272} (1985), no.~3, 399--419. \MR{799670}

\bibitem[Wei73]{Wein}
Peter~J. Weinberger, \emph{On {E}uclidean rings of algebraic integers},
  Analytic number theory ({P}roc. {S}ympos. {P}ure {M}ath., {V}ol. {XXIV},
  {S}t. {L}ouis {U}niv., {S}t. {L}ouis, {M}o., 1972), Amer. Math. Soc.,
  Providence, R. I., 1973, pp.~321--332. \MR{0337902}

\bibitem[Yas10]{bianchi}
Dan Yasaki, \emph{Hyperbolic tessellations associated to {B}ianchi groups},
  Algorithmic number theory, Lecture Notes in Comput. Sci., vol. 6197,
  Springer, Berlin, 2010, pp.~385--396. \MR{2721434 (2012g:11069)}

\end{thebibliography}
}

\end{document}